\documentclass[11pt]{amsart}
\usepackage{amssymb,amsmath,amsthm}
\usepackage[hmargin=3cm,vmargin=3.5cm]{geometry}
\usepackage{xcolor}
\usepackage{graphicx}
\usepackage{tikz}
\usepackage{amscd}
\usepackage{tabularx}
\usepackage{url}

\newtheorem{theorem}{Theorem}[section]

\newtheorem{defn}[theorem]{Definition}

\newtheorem{prop}[theorem]{Proposition}
\newtheorem{cor}[theorem]{Corollary}

\newcommand\A{{\mathcal A}}

\newcommand\Z{{\mathbb{Z}}}
\newcommand\Q{{\mathbb{Q}}}
\newcommand\N{{\mathbb{N}}}

\newcommand{\oplusop}[1]{{\mathop{\oplus}\limits_{#1}}}
\newcommand{\oplusoop}[2]{{\mathop{\oplus}\limits_{#1}^{#2}}}

\newcommand{\uU}{\mathcal{U}}

\def\sbinom#1#2{\left( \hspace{-0.06in}\begin{array}{c} #1 \\ #2
\end{array}\hspace{-0.06in} \right)}
\def\fieldk{\mathbf{k}}
\def\dmod{\mathrm{-mod}}
\def\pmod{\mathrm{-pmod}}
\def\fd{\mathrm{-fd}}
\def\lfd{\mathrm{-lfd}}

\def\fl{\mathrm{-fl}}
\def\fg{\mathrm{-fg}}
\def\pfg{\mathrm{-pfg}}

\def\lra{\longrightarrow}
\def\ra{\rightarrow}
\def\Hom{\mathrm{Hom}}
\def\End{\mathrm{End}}
\def\mc{\mathcal}
\def\Ind{\mathrm{Ind}}
\def\Res{\mathrm{Res}}

\title[Diagrammatic categorification of the Chebyshev polynomials]{Diagrammatic categorification of the Chebyshev polynomials of the second kind}

\author[M. Khovanov]{Mikhail Khovanov}
\address{Department of Mathematics, Columbia University, New York NY 10027}
 \email{khovanov@math.columbia.edu}

\author[R. Sazdanovi\'{c}]{Radmila Sazdanovi\'{c}}
\address{Department of Mathematics, North Carolina State University, Raleigh, NC 27695}
\email{rsazdan@ncsu.edu}

\date{March 27, 2020}

\begin{document}
\maketitle

\begin{abstract}
We develop a diagrammatic categorification of the polynomial ring Z[x], based on a geometrically defined graded algebra.
This construction generalizes to categorification of some special functions, such as Chebyshev polynomials. Diagrammatic
algebras featured in these categorifications lead to the first topological interpretations of the Bernstein-Gelfand-Gelfand
reciprocity property. 
\end{abstract}


\section{Introduction}









When categorifying a vector space with an additional data, such as the structure of a ring or a module over a ring, it's useful to have a bilinear form on the space. Upon categorification, it may turn into the form coming from the dimension of hom spaces between projective objects of the category or from the Euler characteristic of the Ext groups between arbitrary objects. 

Categorification of rings are monoidal categories, with an underlying structure of an abelian or a triangulated category, so one can form the Grothendieck group and then equip it with the multiplication coming from the tensor product on the category. 

One of the simplest rings to consider for a categorification is the polynomial ring $\Z[x]$ in one variable $x$. 
One can try to categorify various inner product on this ring. In \cite{khovanov2015categorifications} the authors considered an example of a  categorification for the inner product $(x^n,x^m)=\sbinom{n+m}{m}$ 

In this paper we will look at a categorification for the inner product corresponding to the Chebyshev polynomials of the second kind \eqref{eq-innpr1}. Monomials
$x^n$ will become objects $P_n$ of an additive category, while the inner product
$(x^n,x^m)$ will turn into the vector space $\Hom(P_n,P_m).$ We choose a
field $\fieldk$ and define $\Hom(P_n,P_m)$ as the $\fieldk$-vector space with the
basis of crossingless matching diagrams in the plane with $n$ points on the left and $m$ points on the right. 

To construct a category out of these vector spaces one 
needs associative compositions
$$ \Hom(P_n,P_m) \otimes \Hom(P_m,P_k) \lra \Hom(P_n,P_k).$$
The standard composition of this kind describes the Temperley-Lieb category. The composition is then given by concatenating diagrams, allowing isotopies rel boundary and removing a closed circle simultaneously with multiplying the diagram by $-q-q^{-1}$. As a special case when $n=m$ one 
recovers the $n$-stranded the Temperley-Lieb algebra $TL_n$ \cite{kauffman1994temperley}. The Temperley-Lieb category allows to extend the Jones polynomial
(which coincides with the one-variable Kauffman bracket) to tangles. 
The Temperley-Lieb (TL) algebra has many idempotents. If we adopt the composition rules of TL algebra, the objects $P_n$ will have many idempotent endomorphisms and
will decompose into a direct sum of smaller objects if the
ambient category is abelian (at least if quantum $[n]!$ is invertible in the ground ring). Either way, they will split into direct summands in the Karoubian envelope of the original additive category. Existence of many direct summands of $P_n$ will make the
Grothendieck group larger than the desired $\Z[x]$ with the
basis of elements $x^n=[P_n]$.

In this paper we consider a different case, which can be viewed as a sort of frozen limit of the Temperley-Lieb category, where diagrams that contain a circle or a pair of U-turns that normally can be straighted up evaluate to zero. Section \ref{sec24} discusses this two-parameter family of categories, isomorphic to the Temperley-Lieb categories at nonzero values of the parameters. Specialization $t=d=0$ produces the monoidal category which  representations are considered in this paper. 

\vspace{0.1in}

{\bf Acknowledgments:}
 The first author's work was partially supported by the NSF grants  DMS-1664240 and DMS-1807425. The second author was partially supported by the Simons Foundation Collaboration grant 318086 and the NSF grant DMS-1854705 during the final stages of this project.

\section{Background and motivation}


\subsection{Idempotented rings and Grothendieck groups}

\label{sec-idempotented}

An idempotented ring $(A, \{1_i\}_{i \in S})$ is a ring $A$, non-unital in general,
equipped with a set of pairwise orthogonal idempotents $\{1_i\}_{i \in S}$
($1_i1_j=\delta_{i,j}1_i$) such that $A= \oplusop{i,j \in S} 1_i A1_j.$
One can visualize an idempotented ring as a generalized matrix algebra, with rows
and columns enumerated by elements of $S$, and the
abelian group $1_iA1_j$ sitting on the intersection of the $i$-th row and $j$-th column.

$$A=\left(
     \begin{array}{ccccc}
        &  & . &  &  \\
        &  & . &  &  \\
        &  & . &  &  \\
     .  & . & 1_iA1_j  & . & . \\
        &  & . &  &  \\
     \end{array}
   \right)
$$

Without loss of generality, one can impose the condition that the idempotents are non-zero (switching between the two versions of the definition amounts to  discarding zero idempotents $1_i=0$).

$A$ is a unital ring if and only if the set of non-zero idempotents in
$S$ is finite; then $1=\displaystyle{\sum_{i\in S}}1_i$. We call $\{1_i\}_{i\in S}$
an \emph{idempotent system}. Forgetting the actual system of idempotents leads to the
notion of a \emph{ring with enough
idempotents}, that is, a ring admitting such a system, see~\cite{angeleri2009locally} for instance.

Idempotented rings can be encoded by preadditive categories. A category is preadditive if for any two objects $i,j$ the set $\Hom(i,j)$
is an abelian group, with bilinear composite maps
\begin{equation*}
  \Hom(i,j) \times \Hom(j,k) \rightarrow \Hom(i,k).
\end{equation*}
An idempotented ring $A$ gives rise to a small preadditive category $\A$ with
objects $i$, over $i \in S$, morphisms from $i$ to $j$ being $1_iA1_j$, and
composition of morphisms coming from the multiplication in $A$:

\begin{equation} \label{eq-multiply}
  1_i A1_j \times 1_j A1_k \stackrel{m}{\lra} 1_i A1_k.
\end{equation}
Vice versa, a small preadditive category $\A$ gives rise to the idempotented ring
\begin{equation*}A= \oplusop{i,j \in Ob(\A)} \Hom_{\A}(i,j).\end{equation*}

 A (left) module $M$ over idempotented ring $A$ is called unital if $M=\oplusop{i \in S} 1_iM.$
In the rest of the paper  by a module we mean a unital left  module unless specified otherwise.
A (left) $A$-module $M$ is called finitely-generated if there exist finitely many
$m_1, \ldots, m_n \in M$ such that $M= A m_1+ A m_2 + \ldots +  A m_n.$ Idempotented
ring $A$ is called (left) Noetherian if any submodule of a finitely-generated
(left) module is finitely generated.  Viewed as a left module over itself, $A$ 
is finitely generated if and only if it is unital, that is, if the system of idempotents is finite, $|S|<\infty$. 

A module is called projective if it is a projective object in the category of
$A$-modules. The notions of unital, projective, finitely-generated module do
not depend on the choice of system of idempotents $\{1_i\}_{i \in S}$.

An idempotent $e \in A$ gives rise to a finitely-generated projective module
$A e$. For any $x \in A$ there exists a finite subset $T \subset S$ such that
$x \in 1_T A 1_T,$ where $1_T:= \displaystyle{\sum_{i \in T} 1_i}$. Equivalently,
$1_T x=x=x 1_T.$ A minimal such $T$ is unique; we can denote it by $T(x)$.

\begin{defn}
Given an algebra $A$ its Grothendieck group $K_0(A)$ is a free abelian group with generators
-- symbols $[P]$ of finitely-generated (left) projective $A$-modules $P$ and defining relations
$[P]=[P_1]+[P_2]$ whenever $P\cong P_1 \oplus P_2$. \end{defn}

There is a canonical  isomorphism between the Grothendieck group of an idempotented ring $A$ and the direct limit
of Grothendieck groups of unital rings $1_TA1_T$, where $T$ ranges over
finite subsets of $S$: 
\begin{equation}
K_0(A) \cong \displaystyle{\lim_{\begin{smallmatrix}
T \subset S\\
 |T| < \infty
\end{smallmatrix}}} K_0 (1_TA1_T),
\end{equation}
with inclusions $T_1 \subset T_2$ giving rise to homomorphisms
$1_{T_1}A 1_{T_1} \ra 1_{T_2} A 1_{T_2}$ and induced maps $K_0(1_{T_1}A 1_{T_1}) \ra
K_0(1_{T_2} A 1_{T_2})$.

For an idempotent $e$, $1_{T(e)}-e$ is an idempotent too, orthogonal to $e$,
\begin{equation*}
  e+(1_{T(e)}-e) = 1_{T(e)}= \displaystyle{\sum_{i \in T(e)}1_i},
\end{equation*}
and $A e \oplus A (1_{T(e)}-e) \cong A 1_{T(e)} \cong \oplusop{i \in T(e)} A 1_i.$
 
 For $i\in S$ define $P_i := A 1_i$. Modules $P_i$ are projective. 

\begin{prop}
  A projective $A$-module $P$ is finitely-generated iff it is isomorphic to a direct
summand of a finite direct sum of projective modules of the form $A1_i$, for some
idempotents $i_1, \dots, i_n\in S$:
\begin{equation}
  P \oplus Q \cong \oplusoop{k=1}{n} A1_{i_k}, \quad i_1, \dots, i_n\in S.
\end{equation}
Equivalently, P is a direct summand of $(A1_T)^m$, for some $m$ and a finite
subset $T \subset S.$
\end{prop}

\begin{proof} $P$ has finitely many generators, and each generator is a finite sum of terms in $1_iP$, over various $i$'s. We can thus assume that each  generator $p_i$ of $P$ is in $1_iP$, for some $i$. There is a map $P_i \lra P$ taking $1_i\in P_i$ to $p_i$. The sum of these maps over all generators of $P$ is a surjective map from $\oplusop{i}P_i$ to $P$. Since $P$ is projective, this map splits. We can now take as $T$ the union of $i$'s and for $m$ the number of generators. \end{proof}

Left, right, and $2$-sided ideals in idempotented rings are defined in the same
way as for usual rings. A left ideal $I \subset A$ is an abelian
subgroup, closed under the left action of $A$, $AI \subset A$. Note that ideals in $A$
respect idempotent decomposition. For instance, a $2$-sided ideal $I$ satisfies: 
\begin{equation*}
  I=\oplusop{i,j \in S} 1_i I 1_j.
\end{equation*}

The center $Z(A)$ of $A$ is defined as the commutative ring of additive
natural transformations of the identity functor
\begin{equation*}
  Id: A\dmod \lra A\dmod
\end{equation*}
(on either the category of left or right $A$-modules). Elements of $Z(A)$ are
in bijection with collections $
  \{x_i | x_i \in 1_iA1_i, \, i\in S, \, {\rm and\,}\, \forall i,j \in S, \, \forall y
\in 1_iA1_j,\, x_iy=yx_j   \}.$

Given a field $\fieldk$, we say that an idempotented ring $A$ is a $\fieldk$-algebra
if the abelian groups $1_i A1_j$ are naturally $\fieldk$-vector spaces, over all
$i,j\in S$, and multiplications (\ref{eq-multiply}) are $\fieldk$-bilinear for
all $i,j,k$. Such $A$ will be called an idempotented $\fieldk$-algebra.

An idempotented $\fieldk$-algebra $A$ is \emph{locally finite-dimensional}
(lfd, for short) if $1_i A 1_j$ are finite-dimensional $\fieldk$-vector spaces
for all $i,j\in S$.

\begin{prop} Finitely-generated modules over an idempotented lfd $\fieldk$-algebra 
have the Krull-Schmidt property.
\end{prop}

\emph{Proof:} The Krull-Schmidt property for a module $M$ is the uniqueness of a 
decomposition of $M$ into a direct sum of indecomposable modules. Let $A$ be an idempotented lfd $\fieldk$-algebra. A sufficient
condition for this property to hold is for $\End_A(M)$ to be a finite-dimensional
$\fieldk$-algebra. Any finitely-generated $A$-module $M$ is a quotient of
a finite direct sum of modules $P_i $.
Let $P$ be such a finite sum surjecting onto $M$. Then $\End_A(M)$ is a
subspace in $\Hom_A(P,M)$. Since $P$ is projective, the natural map
$\End_A(P)\lra \Hom_A(P,M)$  is surjective, and finite-dimensionality of
$\End_A(M)$ follows from finite-dimensionality of $\End_A(P)$.  Algebra $\End_A(P)$ is finite dimensional, since it is 
isomorphic to a finite direct sum of vector spaces of the form $1_i A 1_j$,
which are finite-dimensional.
$\square$

This result, restricted to projective finitely-generated modules, shows that
any such module is a direct summand of $P_i=A1_i$, for some $i$. 

\begin{cor}\label{GrothLFD}
The Grothendieck group $K_0(A)$ of finitely-generated projective  modules  over a locally finite dimensional idempotented $\fieldk$-algebra  $A$ is a free abelian group with a basis given by symbols $[P]$
of indecomposable projective $A$-modules, one for each isomorphism class. \end{cor}

Through the rest of the paper we assume that the idempotented $\fieldk$-algebra $A$ is lfd.

The ring $1_i A 1_i$ is finite-dimensional, hence Artinian. Its
Jacobson radical $J(1_i A 1_i)$ is nilpotent, and the quotient
algebra $1_i A 1_i /J(1_i A 1_i)$ is semisimple. Any idempotent
decomposition of the unit element $1$ in the quotient ring lifts
to a decomposition of $1_i$ in $1_i A 1_i$. For each $i$ choose
a decomposition $1_i = 1_{i,1}+ 1_{i,2} + \dots + 1_{i,r_i}$ of
this idempotent into the sum of primitive mutually-orthogonal
idempotents $1_{i,j} \in 1_i A 1_i$, $1\le j \le r_i$. We refine the
idempotent system $\{ 1_i \}_{i\in S}$ into the idempotent system
made of $ 1_{i,j}$ over all such $i,j$, and denote this
system by $\widetilde{S}$. The idempotented $\fieldk$-algebra $(A,\widetilde{S})$ is also lfd.

Recall that idempotents $e_1, e_2$ in a ring $B$ are called equivalent
if there are elements $x,y\in B$ such that $e_1= xy$, $e_2= yx$.
Idempotents $e_1,e_2$ are equivalent iff projective $B$-modules $Be_1$, $Be_2$
are isomorphic. This notion of equivalence trivially generalizes
to idempotented rings. In particular, some of the idempotents in $A$ might
be equivalent. Idempotents $1_{\widetilde{i}}$, over $\widetilde{i}\in \widetilde{S}$,
decompose into equivalence classes. Choose one representative $i'$ for each
equivalence class, denote the set of such $i'$'s by $S'$, and define
\begin{equation}
A'  \ = \ \oplusop{i',j'\in S'} 1_{i'} A 1_{j'} .
\end{equation}
Idempotented $\fieldk$-algebra $A'$ is a subalgebra of $A$. Its idempotented
system is $\{ 1_{i'}\}_{i'\in S'}$, and $(A', S')$ is lfd.
Algebras $A$ and $A'$ are Morita equivalent. In particular, their categories of
representations $A\dmod$ and $A'\dmod$ are  equivalent and 
their Grothendieck groups are isomorphic.

Idempotented lfd $\fieldk$-algebra $(A',S')$ has the property that all
rings $1_{i'} A' 1_{i'}$ are local, and multiplication (\ref{eq-multiply})
takes tensor product $1_{i'} A' 1_{j'} \otimes 1_{j'} A' 1_{i'}$
into the Jacobson radical of $1_{i'} A' 1_{i'}$ for all $i', j'\in S'$,
$i'\not= j'$. We call such an idempotented lfd algebra \emph{basic} and shorten \emph{basic lfd} to  blfd.
Equivalently, an idempotented lfd $\fieldk$-algebra $(A', S')$ is
basic if projective modules $A'1_{i'}$ are indecomposable and
pairwise non-isomorphic. Moreover, any idempotented lfd $\fieldk$-algebra is Morita equivalent to a basic
one.

\begin{cor} \label{GothBLFD}
The  Grothendieck group $K_0(A)$ of an idempotented blfd $\fieldk$-algebra $(A,S)$  is free abelian, with a basis
 consisting of symbols $[P_i]$ of indecomposable projective modules $P_i = A1_i$.
\end{cor}

Modules $P_i$ are pairwise non-isomorphic, have a unique maximal
proper submodule, and a unique simple quotient, denoted $L_i$.
Module $L_i$ is concentrated in position $i$, in the sense that
$1_i L_i = L_i$, so that $1_j L_i=0$ for $j\not= i$. Any simple $A$-module
$L$ is isomorphic to some $L_i$, for a unique $i$. $\End_A(L_i)$ is a
finite-dimensional division algebra over $\fieldk$.

If $(A,S)$ is an idempotented lfd $\fieldk$-algebra (not necessarily basic),
there is still a bijection between
indecomposable projective $A$-modules $P_u$, labelled by elements of some index set $U$,
and simple modules $L_u$, labelled by elements of the same set, with the property
that $\Hom_A(P_u, L_v)=0 $ unless $u=v$ and $\Hom_A(P_u, L_u)\cong \End_A(L_u,L_u)$.
Thus, $L_u$ is the unique simple quotient of $P_u$. For nonbasic algebras it is
possible for a simple module to be infinite dimensional over $\fieldk$. For instance,
this is true for the idempotented algebra of $S\times S$-matrices with coefficients in $\fieldk$,
with the column module being simple, where $S$ is any infinite set.

For a ring or idempotented ring $A$ denote by $A\fl$ the abelian
category of finite-length left $A$-modules. 

Denote by $G_0(A) = G_0(A\fl)$
the Grothendieck group
of the category of finite-length left $A$-modules. In general, the Grothendieck
group $G_0(\mc{A})$ of an abelian category $\mc{A}$ has generators $[M]$, over all
objects $M$ of $\mc{A}$, and defining relations $[M]=[M_1]+[M_2]$ over all
exact sequences $0\lra M_1 \lra M \lra M_2 \lra 0$.

  To an abelian category $\mc{A}$ we can associate at least three different
versions of the Grothendieck group. The Grothendieck group $G_0(\mc{A})$ has generators - symbols $[M]$ of objects of $\mc{A}$, with short exact sequences as above giving defining relations. 
The Grothendieck group of projective objects $K_0(\mc{A})$ has generators $[P]$, over projective objects $P\in \mc{A}$, and defining relations $[P]=[P_1]+[P_2]$ whenever there is an isomorphism $P\cong P_1\oplus P_2$. 
The split Grothendieck group, which we denote $SG(\mc{A})$, has  with generators $[M]$, over all objects $M$
and defining relations $[M]=[M_1]+[M_2]$ whenever $M\cong M_1\oplus M_2$.

There are obvious homomorphisms
$$ K_0(\mc{A}) \lra SG(\mc{A})  \lra G_0(\mc{A}). $$
The composition $K_0(\mc{A}) \lra G_0(\mc{A})$ is, in general, neither surjective
nor injective.

For an idempotented lfd $\fieldk$-algebra $(A,S)$, there is a bilinear pairing
\begin{equation}\label{eq-pairing}
(\,,) \ : \  K_0(A) \otimes_{\Z} G_0(A) \lra \Z
\end{equation}
given by
\begin{equation}\label{eq-pairing2}
 ([P],[M]) = \dim_{\fieldk}\Hom_A(P,M)
\end{equation}
for a finitely generated projective module $P$ and a finite length module $M$.
If every simple $A$-module is absolutely irreducible, that is, $\End_A(L_u) = \fieldk$
for all $u \in U$, then this pairing is perfect, and the bases $\{[P_u]\}_{u\in U}$
and $\{ [L_u]\}_{u\in U}$ are dual with respect to this pairing. In the absence
of absolute irreducibility the pairing becomes perfect upon tensoring the two Grothendieck groups and $\Z$ with $\Q$.

Let $A$ be an idempotented $\fieldk$-algebra. A left $A$-module $M$ is
called \emph{locally finite-dimensional} (lfd, for short) if
$1_i M$ is a finite-dimensional $\fieldk$-vector space, for any $i\in S$.

Denote by $A\lfd$ the abelian category of lfd $A$-modules (furthermore, it's a thick
subcategory of $A\dmod$). For each $i$ in $S$ there is a homomorphism
\begin{equation}
\rho_i \ : \ G_0(A\lfd) \lra \Z
\end{equation}
taking $[M]$ to $\dim_{\fieldk}(1_i M)$.
Assume now that $A$ is a basic lfd idempotented $\fieldk$-algebra. Then the
image of this homomorphism is spanned by $\dim_{\fieldk}(L_i)\in \Z$, and
the homomorphism is surjective iff $L_i$ is absolutely irreducible.

Taking the product of $\rho_i$ over all $i\in S$ gives a homomorphism
\begin{equation} \label{eq-map-rho}
\rho \ : \ G_0(A\lfd) \lra \prod_{i\in S} \Z .
\end{equation}
If $(A,S)$ is a basic lfd idempotented $\fieldk$-algebra, the image of $\rho$
is the product $\displaystyle \prod_{i \in S} \dim_{\fieldk}(L_i) \Z$ (consider the
image of the object $\oplusop{i\in S} L_i^{n_i} $ of $A\lfd$ for arbitrary
$n_i \in \N$).

If, in addition, all $L_i$'s are absolutely irreducible, $\rho$ is surjective.
It's not clear whether $\rho$ is injective for various natural examples
of lfd idempotented $\fieldk$-algebras, including the ones considered in this paper.

If $A$ is an idempotented lfd $\fieldk$-algebra then any finitely generated
$A$-module is lfd. In particular, simple $A$-modules and finite-length
$A$-modules are lfd, and there are inclusions of categories
\begin{equation}\label{eq-inclusions1}
A\fl \subset A\fg  \subset A\lfd .
\end{equation}
To summarize, 
\begin{itemize}
    \item $A\fl$ is the abelian category of finite-length modules, 
    \item $A\fg$ is the  additive category of 
    finitely-generated modules (abelian category if $A$ is a Noetherian idempotented algebra),
    \item $A\lfd$ is the abelian category  of locally finite-dimensional modules. 
\end{itemize}

\subsection{Chebyshev polynomials}
The Chebyshev polynomials of the second kind $U_n(x)$ are defined
by the recurrence relation and initial conditions
$$ U_{n+1}(x) = 2xU_n(x) - U_{n-1}(x), \ \ U_0(x) = 1,\ \  U_1(x)  = 2x.$$
We will use their rescaled counterparts $ \uU_n = U_n(\frac{x}{2})$, which are
sometimes called \emph{the Chebyshev polynomials of the second kind on the interval}
$[-2,2]$, see~\cite[Section 1.3.2]{mason2002chebyshev}. In this paper we'll simply call $\uU_n$'s
Chebyshev polynomials. The are determined by the recurrence relation
\begin{equation}\label{eq-cheb-rec}
 \uU_{n+1}(x) = x \uU_n(x) - \uU_{n-1}(x)
\end{equation}
and initial conditions
\begin{equation}\label{eq-cheb-init}
 \uU_0(x) = 1 , \ \ \  \uU_1(x) = x .
\end{equation}
For later use, we rewrite the recurrence as
\begin{equation}\label{eq-cheb-rec2}
 x \uU_n(x) =  \uU_{n+1}(x) + \uU_{n-1}(x)
\end{equation}

Chebyshev polynomials $\{\uU_n\}_{n \geq 0}$ form an orthogonal set on the
interval $[-2,2]$ with respect to the weighting function
$\sqrt{4-x^2}$. If we define the inner product on polynomials by

\begin{equation}\label{eq-innpr1}
(f,g)= \frac{1}{2 \pi}  \int_{-2}^2 f(x) g(x) \sqrt{4-x^2} dx,
\end{equation}
then
\begin{equation}
(\uU_n(x),\uU_m(x))= \frac{1}{2 \pi} \int_{-2}^2 \uU_n(x)\uU_m(x)
\sqrt{4-x^2} dx=\left\{%
\begin{array}{ll}
    1 & \hbox{if $n=m$,} \\
    0 & \hbox{otherwise.} \\
\end{array}%
\right.
\end{equation}

Chebyshev polynomials for small values of $n$ are $
    \uU_0(x) =  1, \,
    \uU_1(x) =  x, \,
    \uU_2(x) =  x^2-1, \,
    \uU_3(x) =  x^3-2x, \,
    \uU_4(x) =  x^4-3x^2+1, \,
    \uU_5(x) =  x^5-4x^3+3x, \,
    \uU_6(x) =  x^6-5x^4+6x^2-1, \,
    \uU_7(x) =  x^7-6x^5+10x^3-4x, \,
    \uU_8(x) =  x^8- 7x^6+ 15x^4- 10x^2+1.$

Chebyshev polynomials satisfy the multiplication rule
\begin{equation}\label{eq-cheb-mult}
\uU_n \uU_m  = \uU_{|n-m|} + \uU_{|n-m|+2}+\dots + \uU_{n+m}.
\end{equation}

$\uU_n$ is a monic polynomial of degree $n$ with $n$ real roots
$2 \cos(\frac{\pi k}{n+1}),$ $k=1,\dots,n$. Notice that
$2\cos(\frac{\pi k}{n+1}) = \zeta^k + \zeta^{-k},$ where $\zeta=e^{\frac{\pi i}{n+1}}$.
The $n$-th Chebushev polynomial $\uU_n$ has integral coefficients, which alternate with each change in the
exponent by two.

Chebyshev polynomials can be described via the determinantal formula
\begin{equation}
\uU_n =\begin{vmatrix}
       x & 1 & 0 & \ldots & 0&  0\\
       1 & x & 1 &  \ldots & 0&   0\\
       0 & 1 & x &  \ldots & 0&  0\\
       . & . & . &  \ldots & .& .  \\
       0 &0  & 0  &  \ldots & x&  1\\
       0 &0  & 0 &  \ldots & 1 & x 
\end{vmatrix}
\end{equation}
The Chebyshev polynomial has the following evaluations $\uU_n(1)=0$ and  $\uU_n(-1)=(-1)^{n-1}(n-1)2^{(n-1)}$. Chebyshev polynomials have the generating function
\begin{equation*}
\sum_{n\ge 0} \uU_n(x) t^n = \frac{1}{1-xt+t^2}.
\end{equation*}

\subsection{A categorification of the Chebyshev polynomials via $sl(2)$ representations}
\label{subsec-sl2}

The Lie algebra $sl(2)$ has one irreducible representation in each dimension. Let $V_n$
denote the irreducible $(n+1)$ dimensional representation of $sl(2)$. Representation $V_0$
is trivial, while $V_1$ is the defining (vector) representation of $sl(2)$, and
$V_n \simeq S^n(V_1).$
The category $sl(2)\dmod$ of finite-dimensional $sl(2)$ representations is a semisimple
tensor category
with the Grothendieck ring $K_0(sl(2))$ being a free abelian group with the basis
$\{[V_0],[V_1], \ldots \}$
 in the symbols of all irreducible modules. The multiplication in the Grothendieck
ring is defined by:
 \begin{equation*}
   [V][W]:= [V \otimes W]
 \end{equation*}
 Direct sum decomposition for the tensor product 
 \begin{equation}
     V_n \oplus V_m \simeq V_{|n-m|} \oplus V_{|n-m|+2} \oplus \ldots \oplus  V_{n+m}
 \end{equation}
 categorifies the equation (\ref{eq-cheb-mult}) and gives 
the multiplication in the basis of irreducibles
 \begin{equation*}
   [V_n][V_m]=[V_{|n-m|}]+ [V_{|n-m|+2}]+\ldots + [V_{n+m}].
 \end{equation*}
We identify the Grothendieck ring with the polynomial ring
\begin{equation}
  K_0(sl(2)) \simeq \Z[x]
\end{equation}
in one variable $x$ by taking $[V_1]$ to $x$ and, correspondingly, $[V_1^{\otimes n}]$
to $x^n.$ Under this isomorphism symbols of irreducibles go to
Chebyshev polynomials, $[V_n] \leftrightarrow \uU_n(x)$. Thus, irreducible $sl(2)$ modules
offer a categorification of Chebyshev polynomials.

The Temperley-Lieb category, denoted by $TL$, is a monoidal $\mathbb{C}$--linear category
with objects non-negative integers $n \in \mathbb{Z}_+$, $n \otimes m= n+m,$ and
$\Hom_{TL}(n,m)= \Hom_{sl(2)}(V_1^{\otimes n}, V_1^{\otimes m}).$
A basis in $\Hom_{TL}(n,m)$ is given by the diagrams of crossingless matchings in the
plane between $n$ points on the left and $m$ points on the right.
As a $\mathbb{C}$-linear monoidal category, $TL$ it is generated by morphisms
$V_0 \ra V_1^{\otimes 2}$, $V_1^{\otimes 2} \ra V_0$ that can be depicted by
$\subset$ and $\supset$, with relations shown in Figure~\ref{CHIso}, including the isotopy relations.
\begin{figure}[h] 
\centering
$$\begin{tikzpicture}[very thick, scale = .35]
\begin{scope}[xshift = -13cm]

\draw (-1,0) circle (2cm);
 \draw (3.5,0) node{$=-2$,};
        \end{scope}
 \begin{scope}[xshift = 5cm]
  \draw (-1.5,0) node{$= $}; 
  \draw[-] (0,0) --(3,0); 
  \draw (4,0) node{$=$};
  \draw (5,0) to [out=0, in =180]  
 (8,2) to [out=0, in =90] (11,1) to
  [out=270, in =0] (9,0) to  [out=180, in =90] (8,-1)
  to  [out=270, in =0] (11,-2) to  [out=0, in =180] (14,0)
  ;
     \draw (-12,0) to [out=0, in =180] (-9,-2) to [out=0, in =270] (-6,-1) to
  [out=90, in =0] (-8,0) to  [out=180, in =270] (-9,1)
  to  [out=90, in =180] (-6,2)  to  [out=0, in =180] (-3,0)
  ;
\end{scope}
 \end{tikzpicture}$$
 \caption{Defining relations: on the left, the value of the circle is set to
  two, and isotopy relation is on the right. }\label{CHIso}
\end{figure}
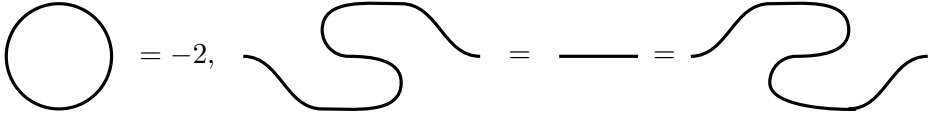

The category $sl(2) \dmod$ of finite-dimensional $sl(2)$-representations  is equivalent
to the Karoubian envelope $C$ of the additive closure of $TL$. To define the latter 
we  first allow finite direct sums of objects of $TL$, than pass to the Karoubian
envelope $C=Kar(Add(TL)).$
Quantum deformation $C_d$ of $C$ is obtained by changing the value of a circle 
to $d=-q-q^{-1}$, $q \in \mathbb{C}$, see Figure~\ref{figure_2} left.

\begin{figure}
$$\begin{tikzpicture}[very thick, scale = .35]
\begin{scope}[xshift = -13cm]

\draw (-1,0) circle (2cm);
 \draw (3.5,0) node{$=d$,};
        \end{scope}
 \begin{scope}[xshift = 5cm]
  \draw (-1.5,0) node{$= t$}; 
  \draw[-] (0,0) --(3,0); 
  \draw (4,0) node{$=$};
  \draw (5,0) to [out=0, in =180]  
 (8,2) to [out=0, in =90] (11,1) to
  [out=270, in =0] (9,0) to  [out=180, in =90] (8,-1)
  to  [out=270, in =0] (11,-2) to  [out=0, in =180] (14,0)
  ;
     \draw (-12,0) to [out=0, in =180] (-9,-2) to [out=0, in =270] (-6,-1) to
  [out=90, in =0] (-8,0) to  [out=180, in =270] (-9,1)
  to  [out=90, in =180] (-6,2)  to  [out=0, in =180] (-3,0)
  ;
\end{scope}
 \end{tikzpicture}$$ 
 \caption{Defining relations in $T_{d,t}$}\label{figure_2}
\end{figure}
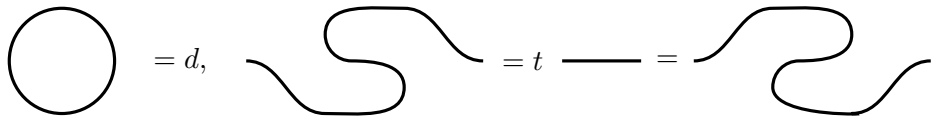

This results in the Temperley-Lieb category $TL_d$. This category has non-negative numbers $n$ as objects and $k$-linear combinations of diagrams of planar arcs  with $n$ left and $m$ right endpoints as morphisms, modulo the isotopy relations and evaluating a circle to $d$. Allowing finite direct sums of objects and passing to the Karoubi envelope results in the 
monoidal category $C_d=Kar(Add(TL_d)).$

For $q$ not a root of unity, $C_d$ is equivalent
to the category of finite-dimensional representations of the quantum group $U_q(sl(2))$
on which the Cartan generator $K$ acts with eigenvalues powers of $q$.

We refer the reader to Queffelec-Wedrich \cite{queffelec2018extremal} for a similar categorification of the Chebyshev polynomials of the first kind. 

\subsection{A 2-parameter family of Temperley-Lieb type monoidal categories}\label{sec24}

One-parameter family of Temperley-Lieb categories can be viewed as a limit of a two-parameter family $TL_{d,t}$ of monoidal categories, where we evaluate a circle to $d$ and a squiggle to $t$ times the identity, see the Figure~\ref{figure_2}.  

\begin{figure}
$$\begin{tikzpicture}[very thick, scale = .35]
\begin{scope}[xshift = -13cm]

\draw (-1,0) circle (2cm);
 \draw (3.5,0) node{$=dt^{-1}$,};
        \end{scope}
 \begin{scope}[xshift = 5cm]
  \draw (-1.5,0) node{$=$}; 
  \draw[-] (0,0) --(3,0); 
  \draw (4,0) node{$=$};
  \draw (5,0) to [out=0, in =180]  
 (8,2) to [out=0, in =90] (11,1) to
  [out=270, in =0] (9,0) to  [out=180, in =90] (8,-1)
  to  [out=270, in =0] (11,-2) to  [out=0, in =180] (14,0)
  ;
     \draw (-12,0) to [out=0, in =180] (-9,-2) to [out=0, in =270] (-6,-1) to
  [out=90, in =0] (-8,0) to  [out=180, in =270] (-9,1)
  to  [out=90, in =180] (-6,2)  to  [out=0, in =180] (-3,0)
  ;
\end{scope}
 \end{tikzpicture}$$ 
 \caption{Rescaling relations in $T_{d,t}$ for $t\not= 0$ to those of $T_{dt^{-1}}$}\label{figure_3}
\end{figure}
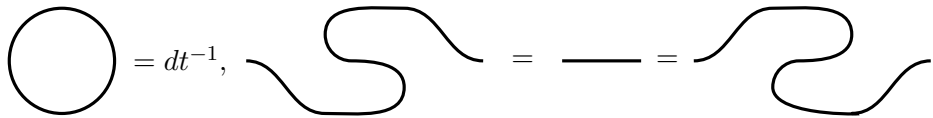

If $t$ is a nonzero element of the ground field $\fieldk$, one can rescale by dividing either left or right turns by $t$ to reduce to $t=1$ case and the relations shown in Figure \ref{figure_3}.

These are the relations in the Temperley-Lieb category with the value of the circle $dt^{-1}$. Consequently, there's an equivalence (even an isomorphism) of categories 
$$ TL_{d,t} \ \cong \ TL_{dt^{-1}}$$ 
between $TL_{d,t}$ and the original Templerley-Lieb category with a circle evaluating to $dt^{-1}$. 

The only case when we cannot divide by $t$ over a field is when $t=0$. This is a degenerate case, which further splits into two cases $d\not= 0 $ and $d= 0$. 

If $d\not= 0$, rescaling a left or a right return by $d^{-1}$ reduces to the case $d=1$, giving an equivalence of categories
$$ TL_{d,0}\ \cong \ TL_{1,0}, \ \  \mathrm{if} \ \ d\not= 0.$$ 

In the present paper we deal with the last case, 
$d=t=0$, which is the maximally degenerate one, in some sense. 

The following provides a simple summary of two-parameter family $TL_{d,t}$ based on possible values of $d$ and $t$:
\begin{itemize}
    \item $t \neq 0$: one can rescale to get the  Temperley-Lieb category $TL_{dt^{-1}}$,
     \item $t=0$ provides two options:
     \begin{enumerate}
    \item$d \neq 0$: rescale to $d=1$, which yields the  category $TL_{1,0}$, or 
        \item$d=0$ the category analyzed in this paper.
         
     \end{enumerate}
\end{itemize}

Note that $d=1,t=0$ case is somewhat reminiscent of the categorification of the polynomial ring in \cite{khovanov2015categorifications}. Since the value of the circle is one, maps of objects $n\lra n+2$ given by inserting a return somewhere inside the diagram of $n$ parallel lines has a splitting given by the reflected diagram, so that $n$ becomes a direct summand of $n+2$ in the Karoubi additive closure of $TL_{1,0}$. There are $n+1$ such maps, one for each position of a return relative to $n$ parallel lines. Due to squiggles being $0$, these maps and their duals are mutually orthogonal, and allow to split off $n+1$ copies of the object $n$ from the object $n+2$ in $C_{1,0}=Kar(Add(TL_{1,0}))$. Iterating, object $2n$  contains the $n$-th Catalan number of copies of object $0$ as direct summands.



\section{Diagrammatic categorification of the Chebyshev polynomials of
the second kind}

Let ${}_n^{}B_m^c$, for $n,m \geq 0$ denote the set of
isotopy classes of plane diagrams consisting of $n$ vertices on
the line $x=0$ and $m$ vertices on the line $x=1$,
and crossingless connections between them (no intersections or
self-intersections are allowed). Crossingless connections fall into three types:
through arcs that connect points on lines $x=0$ and $x=1$, and left and right returns
 that connect pairs of points on the line $x=0$ and line $x=1$, respectively, see Figure~\ref{CHBasis}. The number of through arcs is called the \emph{width}
of a diagram. We let ${}_n^{}B_m^c(w)$ and ${}_n^{}B_m^c(\leq w)$ denote all diagrams
in ${}_n^{}B_m^c$ of width exactly $w$ and at most $w$, respectively. Denote by
$B_m^c$ the disjoint union of sets ${}_n^{}B_m^c$ over all $n\ge 0$, likewise
for ${}_n^{}B^c$.

\begin{figure}[h]
\centering
\includegraphics[width=0.5\textwidth]{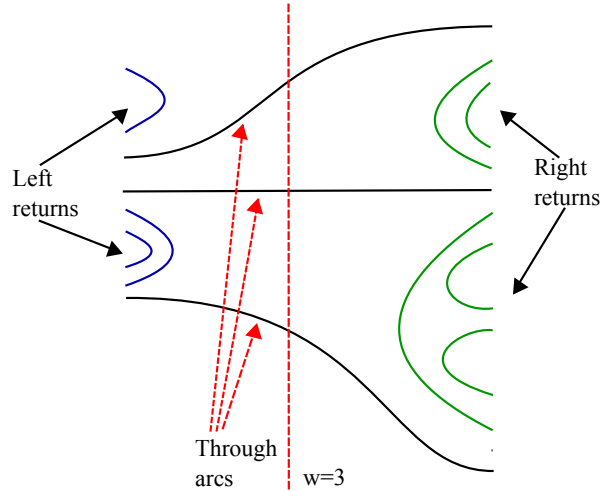}
 \caption{A diagram in ${}_9B_{13}^c(3)$.}
 \label{CHBasis}
\end{figure}

Moreover, through arcs are not allowed to have critical points
with respect to the orthogonal projection onto the $x$--axis, and returns need
to have exactly one critical point, see Figure ~\ref{CHIso}. Diagrams are considered up to
isotopies that preserve these conditions.

The cardinality of the set ${}_n^{}B_m^c$ is the $k$-th
Catalan number $C_{k}=\frac{1}{k+1} \sbinom{2k}{k}$ for $k=\frac{n+m}{2}.$\\
Let $A^c=\oplusoop{n,m \geq 0}\, {}_n^{}A_m^c$, where ${}_n^{}A_m^c$
denotes a vector space over a base field $\fieldk$ with the basis of diagrams
${}_n^{}B_m^c$. In addition, $A^c$ is given $\fieldk$-linear
multiplication by the horizontal concatenation of diagrams. More precisely, product
$xy \in {}_n^{}B_s^c$ of diagrams $x \in {}_n^{\,}B_m^c$ and $y \in {}_lB_s^c$  is zero unless $m=l$, and also equals zero if the resulting diagram contains
one of the diagrams (other then the horizontal line) shown on Figure ~\ref{CHIso}. In other words, the product is zero if there is a circle in the concatenated diagram, or a twist which is a composition of two returns. Equivalently, it's zero if some connected component of the concatenation has more than one critical point under the projection to the $x$-axis. 

$A^c$ equipped with this multiplication becomes an associative non-unital $\fieldk$--algebra.
This algebra contains an idempotent $1_n$ for each $n \in \mathbb{Z}_+$,
consisting of $n$ through arcs. These elements $\{1_n\}_{n \geq 0}$
satisfy the following equalities:
\begin{eqnarray*}
1_nx &=& x, {\, \rm for\,}x \in {}_n^{}B_m^c,\\
y 1_n&=& y, {\, \rm for\,} \,y \in {}_l^{}B_n^c,\\
1_m1_n&=& \delta_{m,n}\,1_n.
\end{eqnarray*}
$\{1_n\}_{n \geq 0}$ are mutually-orthogonal idempotents. For any $x\in A^c$, there exists
$k \in \mathbb{N}$
such that $\displaystyle{\sum_{n=0}^{k}}1_nx=x=x\displaystyle{\sum_{n=0}^{k}}1_n$.
Thus, $A^c$ is an idempotented $\fieldk$-algebra. Alternatively, this structure can be viewed
as a preadditive category with objects $n \in \Z_+$, morphisms $n \to m $ being
${}_n^{\,} A^c_m$, and composition--the product in $A^c$.


Next we define types of modules of interest to us. Let $P_n=A^c 1_n$ denote the module
over $A^c$ generated (as a vector space)  by all diagrams in $B^c$ with $n$ right endpoints.
$P_n$'s are indecomposable projective modules and $ A^c \cong \oplusop{n\ge 0} P_n.$
Any projective $A^c$-module is a direct sum of $P_n$'s, with multiplicities being
invariants of the module. The Grothendieck group $K_0(A^c)$ is free abelian with the 
basis $\{[P_n]\}_{n\ge 0}$.

\begin{theorem} \label{ThmPnProjC}
Any projective left locally finite--dimensional $A^c$-module $P$ is isomorphic to a finite direct sum of indecomposable projective modules $P_n$, $P \cong \oplusop{n\ge 0} P_n^{a_n},$ where the multiplicities $a_n \in \Z_+$ are invariants of $P$.
\end{theorem}

\begin{proof} The theorem follows from more general results about projective modules over idempotented lfd algebras discussed in Section~\ref{sec-idempotented} and the observation that $1_n A^c1_n$ is a local algebra, with the maximal nilpotent ideal $J_n$ spanned by diagrams other than the identity diagram.  The quotient $1_nA^c1_n/J_n$ is the ground field $\fieldk$. Furthermore, the multiplication map $1_n A_c 1_m \otimes 1_m A_c 1_n \lra 1_n A_c 1_n$ for $m\not= n$ has the image in $J_n$. 
\end{proof}

\begin{figure}[h]
\centering
\includegraphics[width=0.5\textwidth]{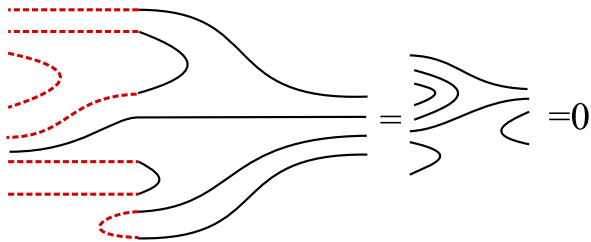}\caption{Action of
algebra $A^c$ on standard modules $M_n$: if a resulting diagram
contains a right return it is equal to zero.} \label{CHActionMn}
\end{figure}


Indecomposable projective module $P_n$ has a unique simple quotient module $L_n$. This module is one-dimensional. Denote the  the basis vector generating $L_n$ by $1_n$  and note that, with some informality, we use the same
notation for this basis vector as for the corresponding idempotent in $A^c$.
Any diagram from $B^c_n$ other than $1_n$ acts by zero on $L_n$.
Any simple $A^c$-module is isomorphic to $L_n$, for
a unique $n$. The idempotented $\fieldk$-algebra $A^c$ with this set of idempotents $\{1_n \}_{n \geq 0}$
is basic locally finite-dimensional,
providing an example to the theory discussed in Section~\ref{sec-idempotented}.
All of its simple modules are absolutely irreducible. The bilinear
pairing (\ref{eq-pairing}), in the case of $A^c$, is perfect. The bases
$\{[P_n]\}_{n\ge 0}$ and $\{[L_n]\}_{n\ge 0}$ of $K_0(A^c)$ and $G_0(A^c)$, respectively, are dual to each other.


Let $M_n$, for $n \geq 0$ denote the standard module, a quotient of $P_n$ by the
submodule $I_n$ spanned by diagrams with $n$ right endpoints
and at least one right return. While $P_n$ has a basis of all diagrams
in $B^c_n$, diagrams with no right returns give a basis in $M_n$.
Any diagram which contains a right return acts by zero on $M_n$
(irregardless  of the number of right endpoints of this diagram),
see Figure~\ref{CHActionMn}, for instance.

Examples of diagrams corresponding to basis
elements in projective, standard and simple modules are shown in Figure~\ref{CHPnMnLn}.

\begin{figure}[h]
\centering
\includegraphics[width=0.8\textwidth]{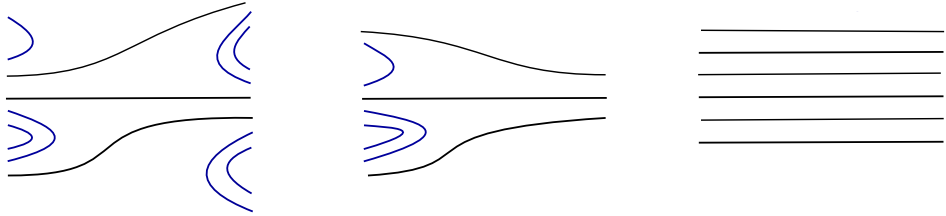}
\caption{Typical basis elements of projective, standard and
simple $A^c$--modules, respectively.} \label{CHPnMnLn}
\end{figure}


$A^c$ is not Noetherian, since the submodule $Q$ of the
projective module $P_0$ generated by the diagrams $b_i$, $i>0$,
shown in Figure  ~\ref{CHnotNoetherian} is infinitely generated:
$b_i$ does not belong to the submodule of $Q$ generated by
diagrams $b_j$, $j<i$, 
since the number of left returns parallel to the outermost return can not be increased by
the left action of $A^c.$
\begin{figure}
\centering
\includegraphics[width=0.6\textwidth]{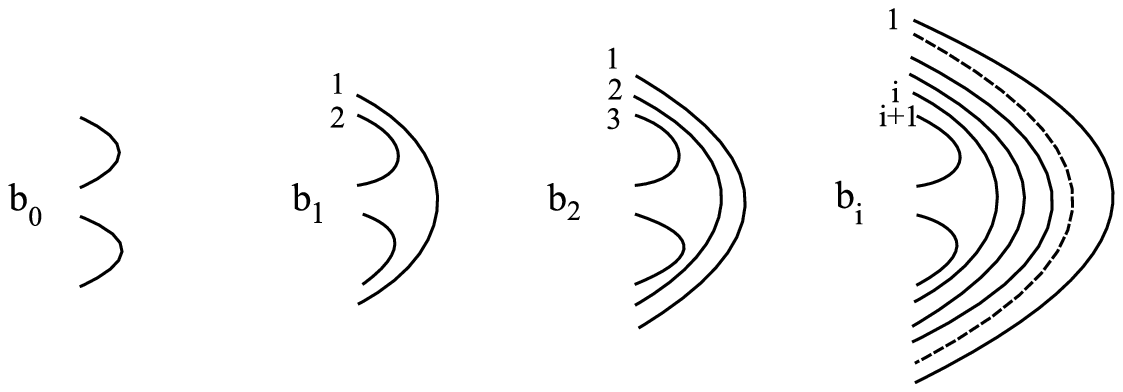}
\caption{Generators of the submodule $Q$ of $P_0$ which can not be
finitely generated.} \label{CHnotNoetherian}
\end{figure}

How to circumvent the problem of $A^c$ not being Noetherian, i.e., that
the category of finitely generated $A^c$--modules is not abelian? One
solution is to consider the abelian category $A^c \fl$ of
finite-length $A^c$-modules. Since simple $A^c$-modules are all one-dimensional,
$A^c\fl$ is also the category of finite-dimensional $A^c$-modules.
Also, by analogy with the slarc algebra~\cite{ SazdanovicRadmila2010Coka, khovanov2015categorifications} and Section~\ref{sec-idempotented},
we will use the category $A^c \lfd$ of locally finite--dimensional $A^c$--modules
and the category $A^c\pfg$ of projective finitely-generated modules.


The Grothendieck group $G_0(A^c)= G_0(A^c\fl)$
 is free abelian with a basis given by symbols of
simple modules $[L_n]$, $n \geq 0.$ The category $A^c \lfd$ is abelian as well,
and contains modules $L_n$, $M_n$ and $P_n$.
However, the Grothendieck group of $A^c \lfd$ is large, admitting a
surjection
\begin{eqnarray*}
 G_0(A^c \lfd) \stackrel{\rho} \rightarrow \prod_{n \geq0} \mathbb{Z}
\end{eqnarray*}
given by sending
$[M]  \mapsto  (\dim 1_0 M, \dim 1_1 M, \ldots), $ see equation \eqref{eq-map-rho}
in Section~\ref{sec-idempotented}. 
Observe that  $[P_n]$, over all $n \geq 0$, are
linearly independent in $G_0(A^c \lfd)$, so that the natural map
$K_0(A^c) \ra G_0(A^c \lfd)$ is injective. Overall, there are natural
injective maps of Grothendieck groups
$$ K_0(A^c) \lra G_0(A^c\lfd) \longleftarrow G_0(A^c), $$
and a perfect pairing $K_0(A^c) \displaystyle{\otimes_{\Z}} \ G_0(A^c \fd) \ra \Z$ with
$$([P],[M])= \displaystyle{\dim \Hom_{A^c}(P,M)},$$
making bases $\{[P_n]\}_{n \geq 0}$ and
 $\{[L_n]\}_{n \geq 0}$ dual to each other, $([P_n],[L_m])= \delta_{n,m}.$


\begin{theorem}
There is a natural isomorphism sending generators $[P_n]$ of $K_0(A^c)$ to $x^n \in \Z[x]$, categorifying $\Z[x]$
as an abelian group:
\begin{equation}  K_0(A^c) \cong \Z[x].
\label{CHCatP}
\end{equation}
\end{theorem}

In order to categorify $\Z[x]$ as a ring we need to define a monoidal structure
on our category.
The Hom space $\Hom (P_n, P_m)$ between projective modules $P_n$ and $P_m$ has a
basis of diagrams in ${}_n^{}B^c_m$. Stacking diagrams on top of each
other defines bilinear maps:

\begin{equation}\Hom(P_{m_1},P_{n_1}) \otimes \Hom(P_{m_2},P_{n_2})
\ra \Hom(P_{m_1+m_2}, P_{n_1+n_2}) . \label{CHHoms}
\end{equation}

Define the tensor product  bifunctor $A\pmod \otimes A\pmod \ra A\pmod$ on objects by:
\begin{equation}P_{n_1} \otimes P_{n_1}:= P_{n_1+n_2}
\end{equation}
and on morphisms $P_{m_1}\stackrel {\alpha}\ra P_{n_1}$, $P_{m_2}\stackrel
{\beta}\ra P_{n_2}$, as in \cite{SazdanovicRadmila2010Coka,khovanov2015categorifications}, by stacking basic morphisms
one on top of another and extending using bilinearity
$(\alpha,\beta) \mapsto \alpha \otimes \beta$ where $P_{m_1+m_2}\stackrel {\alpha \otimes
\beta}\longrightarrow P_{n_1+n_2}.$

On the level of Grothendieck groups, the tensor product descends to  the  multiplication, with 
$[P_n]=x^n$ and $[P_n \otimes P_m]=[P_{n+m}]=x^nx^m=x^{n+m}.$ This gives us a categorification of the ring $\Z[x]$ via the category of finitely-generated projective $A^c$-modules. 
Notice that $\otimes$ is not symmetric i.e. $M\otimes N \ncong\ N \otimes M.$

The tensor product extends to the category of complexes of projective modules up to homotopies.
Denote by $\mc{C}(A^c \pfg)$ the category of bounded complexes of (finitely generated) projective
$A^c$--modules modulo homotopies (given two chains, say $C_1, C_2 \in C(A^c \pmod)$
then $C_1 \otimes C_2$ is $\displaystyle{\oplusop{i,j \in \mathbb{Z}}C_1^i \otimes C_2^j}$
and the differential $d=d_1 \otimes 1 +1 \otimes
d_2$), see \cite{gelfand2013methods, weibel1995introduction}.  The tensor product is a bifunctor
$$ \otimes \ : \ \mc{C}(A^c \pfg) \times \mc{C}(A^c \pfg) \lra  \mc{C}(A^c \pfg).$$

Denote by $\widetilde{X}_{n,n+2k}$ the subset of ${}_n^{\,} B^c_{n+2k}$ consisting of
diagrams without left returns, and denote by $X_{n,n+2k}$ the cardinality
of $\widetilde{X}_{n,n+2k}$. Note that
$$X_{n,n+2k}=\frac{n+1}{n+k+1}\sbinom{n+2k}{k}.$$
We also let $X_{n,m}=0$ unless $m-n$ is an even nonnegative integer.

Let $P_n(\leq w)$ denote a submodule of $P_n$ generated by all diagrams
in $P_n$ of width less than or equal to $w,$ for $w \geq 0$.
$P_n$ has a finite decreasing
filtration
\begin{equation*}
P_n = P_n(\leq  n)\supset P_n(\leq  n-2) \supset \ldots
\supset P_n(\leq  n-2k) \supset \ldots  \supset P_n(\leq  n- 2\lfloor\frac{n}{2}\rfloor).
\end{equation*}
The quotient module
$P_n(\leq  n-2k)/P_n(\leq  n-2(k+1))$ has a natural basis consisting of
diagrams in $B^c_n$ whose width is exactly $n-2k.$
Any such diagram can be uniquely presented as a composition of a diagram in
$B^c_{n-2k}$ with no right returns (corresponding to a basis element of the standard
module $M_{n-2k}$), and a diagram in $\widetilde{X}_{n-2k,n}$ which has exactly
$k$ right and no left returns, see Figure \ref{CHFilterPn}.
For $k \leq \lfloor\frac{n}{2}\rfloor$ we have
\begin{equation}
X_{n-2k,n}= \frac{n-2k+1}{n-k+1}\sbinom{n}{k}.
\end{equation}

Each diagram in $\widetilde{X}_{n-2k,n}$ (one is shown in Figure~\ref{CHFilterPn} on the right)
provides a generator for a copy of $M_{n-2k}$, viewed as a direct summand of
the quotient module $P_n(\leq  n-2k)/P_n(\leq  n-2(k+1))$.
This yields the following relation:

\begin{equation*}
  P_n(n-2k)/P_n(n-2(k+1))\cong M_{n-2k}^{\widetilde{X}_{n-2k,n}},
\end{equation*}
where $M^I$, for a module $M$ and a set $I$, denotes the direct sum of copies of $M$,
one for each element of the set $I$. For $i\in I$ we denote the corresponding copy
of $M$ in the summand by $M^i$.

We get a relation in any suitable Grothendieck group of $A^c$, for
instance in $G_0(A^c \lfd)$:
\begin{equation}\label{eq-pn-via-ms}
 [P_n] \ = \ \sum^{\lfloor \frac{n}{2} \rfloor}_{k=0}
X_{n-2k, n}[M_{n-2k}].
\end{equation}

Module $M_m$ appears $X_{m,n}$ times in the above filtration of $P_n$ by standard
modules. This multiplicity number is independent
of the choice of a filtration of $P_n$ with subsequent
quotients isomorphic to standard modules. We denote this multiplicity by
$[P_n:M_m] = X_{m,n}$.

\begin{figure}[h]
\centering
\includegraphics[height=0.3\textwidth]{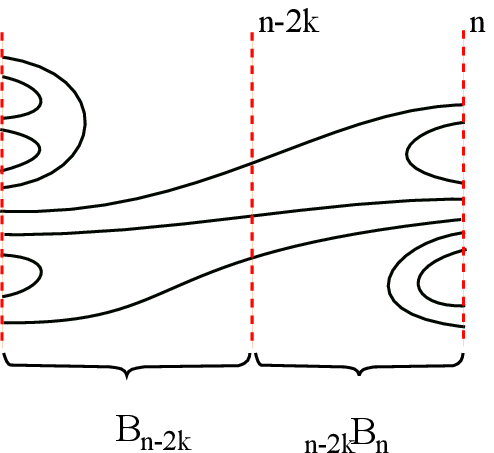}\hspace{1cm}
\includegraphics[height=0.25\textwidth]{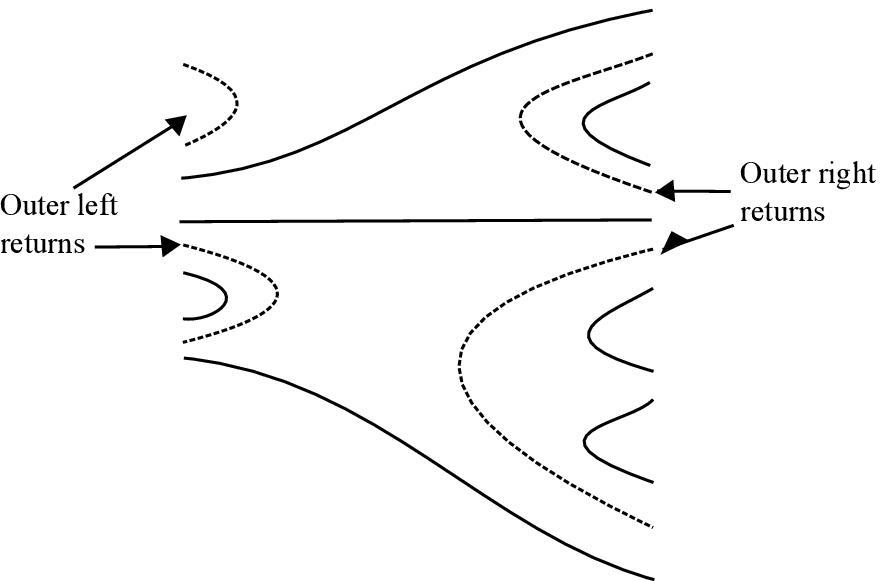}
\caption{Left: Decomposition of a diagram in $B^c_n$
into a diagram in $B^c_{n-2k}$ with no right returns (corresponding to a basis element
of the standard module $M_{n-2k}$), and a diagram in $\widetilde{X}_{n-2k,k}$. Right: Diagram in $B^c$ with nested returns; outer returns are represented
 by dashed lines.}
\label{CHFilterPn}
\end{figure}

 Diagrams in $B^c$ can have nested returns, for example a diagram contains right
(or left) nested returns if endpoints of one return lie between the endpoints
of the other, outer return, see the right picture on Figure \ref{CHFilterPn}.


Let $X$ and $Y$ be upper-triangular $\mathbb{N} \times\mathbb{ N} $ matrices
with nonzero entries
\begin{equation}\label{eq-x-and-y}
X_{n,n+2k}=\frac{n+1}{n+k+1}\sbinom{n+2k}{k} , \quad 
 Y_{n,n+2k}=(-1)^k \sbinom{n+k}{k},  n,k\ge 0.
\end{equation}
The entries of $X$ count diagrams in $\widetilde{X}_{n,n+2k}$, as defined earlier.
Let $\widetilde{Y}_{n,n+2k}$ be the set of diagrams in ${}_n^{}B^c_{n+2k}$ with
no left returns and only unnested right returns. Such diagrams
necessarily has $k$ right returns, and the absolute value of $Y_{n,n+2k}$ is
the cardinality of the set $\widetilde{Y}_{n,n+2k}$.

\begin{prop} Matrices $X$ and $Y$ are mutually inverse, $XY= YX= \mathrm{Id}$.
\end{prop}
\emph{Proof:}
Composing a diagram from $\widetilde{Y}_{n,n+2k}$ with a diagram from
$\widetilde{X}_{n+2k,n+2m}$ for $0\le k\le m$ results in a diagram
in $\widetilde{X}_{n,n+2m}$. A given diagram $\gamma\in\widetilde{X}_{n,n+2m}$
has $2^r$ such presentations as a composition (with varying $k$), where $r$
is the number of outer (right) returns of $\gamma$, that is, returns that
are not nested inside other returns. For instance, diagram (a) in
Figure~\ref{CHResMn} has two outer returns and $2^2$ possible decompositions,
two of which are shown as diagrams (b) and (c). Each diagram $\gamma$
contributes to the  $(n,n+2m)$-entry of
the matrix $YX$ (the entry is the sum of contributions from all such diagrams).
The contribution is the alternating sum of $2^r$ one's,
with signs counting the number of outer returns of $\gamma$ that appear
in the $Y$-part of the decomposition. Clearly, the contribution is $0$, unless
$r=0$, which is only possible with $m=0$, in which case the entry is $1$. We
see that $YX=\mathrm{Id}$, while $XY=\mathrm{Id}$ follows now from
upper-triangularity of $X$ and $Y$.
$\square$

\begin{figure}[h]
\centering
\includegraphics[width=0.7\textwidth]{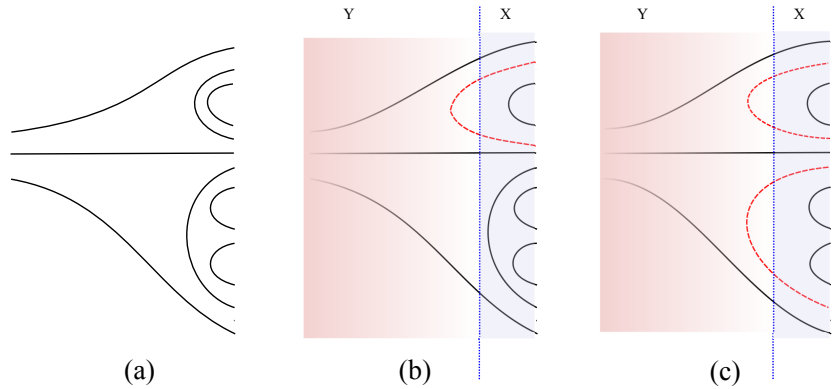}
\caption{Diagram (a) contributes to $(3,13)$-entry of matrix $YX$. Diagrams (b) and (c) are two out of four of its decompositions.} \label{CHResMn}
\end{figure}


On the level of Grothendieck groups, inverting (\ref{eq-pn-via-ms}), we get
\begin{equation}\label{eq-mn-via-ps}
 [M_n] \ = \  \sum^{\lfloor\frac{n}{2}\rfloor}_{k=0}
 Y_{n-2k,n}[P_{n-2k}] \ = \ \sum^{\lfloor\frac{n}{2}\rfloor}_{k=0} (-1)^k \sbinom{n-k}{k}
 [P_{n-2k}].
\end{equation}
Based on this equation, one expects to have a projective resolution of the standard
module $M_n$ with the $k$-th term equal to the direct sum of $|Y_{n-2k,n}|$
copies of   $P_{n-2k}$:
\vspace{-0.3cm}
\begin{equation}
0 \ra \ldots \ra P_{n-2k}^{\small{\sbinom{n-k}{k}}} \stackrel{d_k}\lra
\ldots \stackrel{d_2}\lra P_{n-2}^{n-1} \stackrel{d_1}\lra
P_n \ra M_n \ra 0.
\label{CHResMnbyPn}
\end{equation}

Let us now construct such complex, parametrizing summands of the $k$-th term by
elements of $\widetilde{Y}_{n-2k,n}$:
\begin{equation}
0 \ra \ldots \ra P_{n-2k}^{\widetilde{Y}_{n-2k,n}} \stackrel{d_k}\lra
P_{n-2(k-1)}^{\widetilde{Y}_{n-2(k-1),n}} \stackrel{d_{k-1}}\lra
\ldots \stackrel{d_{2}}\lra P_{n-2}^{\widetilde{Y}_{n-2,n} } \stackrel{d_{1}}\lra
P_n \ra M_n \ra 0,
\label{CHResMnbyPnTilde}
\end{equation}
with the convention that $\widetilde{Y}_{n,n}$ is the one-element set consisting of the diagram $1_n$.

The module map $ P_{n-2k}\lra P_{n-2(k-1)}$ can be described uniquely by a
linear combination $c$ of diagrams in ${}_{n-2k}^{}B^c_{n-2(k-1)}$. A linear combination $c$
takes $a\in P_{n-2k}$ to $ac\in P_{n-2(k-1)}$.
The differential $d_k$ in \eqref{CHResMnbyPnTilde} decomposes as the sum of its components
$$d_{k,\beta,\alpha} \ : \ P_{n-2k}^{\alpha} \lra P_{n-2(k-1)}^{\beta}$$
over all $\alpha\in \widetilde{Y}_{n-2k,n}$ and  $\beta \in \widetilde{Y}_{n-2(k-1),n}$.

\begin{figure}[h]
\centering
\includegraphics[width=0.7\textwidth]{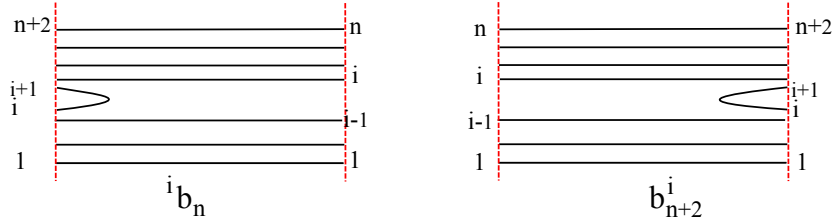}
\caption{Diagrams used in defining differentials in resolutions of
standard and simple modules.} \label{CHDiffMnLn}
\end{figure}

Composing the diagram $b^i_{n-2(k-1)}$ with $\beta$ gives us a diagram $b^i_{n-2(k-1)} \beta
\in {}_{n-2k}^{}B^c_n$. This diagram has $k$ right returns, no left returns, and lies in
$ \widetilde{Y}_{n-2k,n}$ iff it has no nested returns. The unique right return of
$b^i_{n-2(k-1)}$ becomes a right return of the composition. Assuming that the composition
has no nested returns, let $j$ be the order of this right return, where we count right
returns from top to bottom.

For each $\alpha$ and $\beta$
as above there is at most one $i$ such that $\alpha = b^i_{n-2(k-1)} \beta$.
If there is no such $i$, we set $d_{k,\beta,\alpha}=0$, otherwise set
$d_{k,\beta,\alpha}(a) = (-1)^{j-1} a b^i_{n-2(k-1)}$ for $a\in P^{\alpha}_{n-2k}$.
Differential $d_k$ is the sum of $d_{k,\beta,\alpha}$ over all $\beta, \alpha$ as above.
Thus, the differential is a signed sum of maps given by composing with diagrams
with no left and one right return. The equation $d_{k-1}d_k=0$ follows at once.

\begin{figure}[h]
\centering
\includegraphics[width=0.6\textwidth]{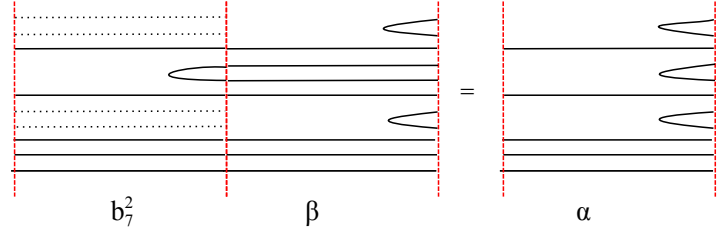}
\caption{One component of the differential sends $P_5^{\alpha}$ into $P_7^{\beta}$
for these $\alpha\in {}_5B^c_{11}$ and $\beta\in {}_7B^c_{11}$
by composing with $-b^2_7$; we have $b^2_7\beta=\alpha$. } \label{CHDiffDefPnMm}
\end{figure}

\begin{prop} \label{prop-res-stand}
The complex \eqref{CHResMnbyPnTilde} is exact, giving the finite projective resolution of standard modules $M_n$, $n>0$.
\end{prop}
\begin{proof}
Diagrams $a\in B^c_{n-2k}$ constitute a basis of the module $P_{n-2k}^{\alpha}$.
Composition $a\alpha$ is either $0$ or a diagram in $B^c_n$. Components of the differential
$d_k$ take $a$ to a signed sum of diagrams $ab^i_{n-2(k-1)}$, which are basis elements
in $P_{n-2(k-1)}^{\beta}$, with the property that $a\alpha= ab^i_{n-2(k-1)}\beta.$
Thus, the product $a\alpha$ is "preserved" by the differential, in
the following sense. The complex (\ref{CHResMnbyPnTilde}) is a direct sum of complexes
of vector spaces, over all $b\in B^c_n,$ which have as basis pairs $(\alpha,a)$
over all $\alpha,a$ as above with $a\alpha=b$. Each of this complexes of vector
spaces is isomorphic to the complex given by collapsing an anticommutative $r$-dimensional
cube of one-dimensional vector spaces, one for each vertex of the cube, with
all edge maps being isomorphisms, where $r$ is the number of right returns of $b$.

If $r>0$, the corresponding complex is contractible, while when $r=0$ the complex
has one-dimensional cohomology in degree $0$. Thus, cohomology of (\ref{CHResMnbyPnTilde})
lives entirely in cohomological degree $0$ and has a basis of diagram in $B^c_n$
without right returns. These diagrams constitute a basis of $M_n^c$, implying that
(\ref{CHResMnbyPnTilde}) is a resolution of the standard module $M_n$.
\end{proof}

\begin{figure}
\centering
\includegraphics[width=0.5\textwidth]{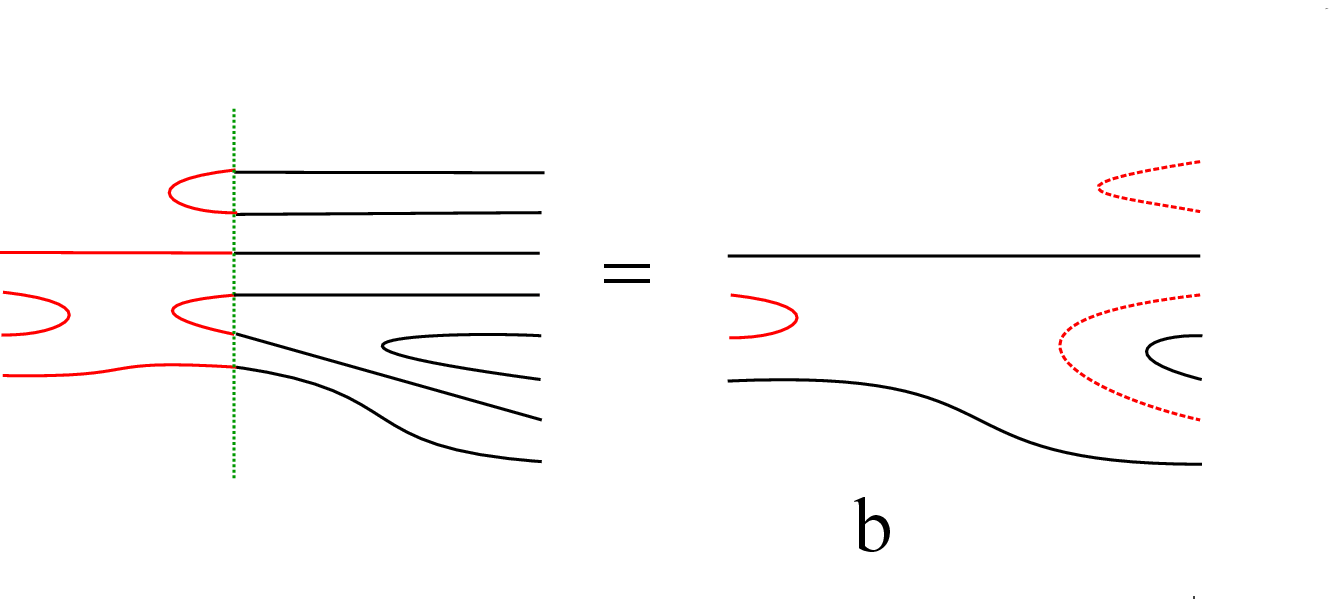}
\caption{Diagram $b$ on the right gives rise
to a $2$-dimensional cube since it has $2$ outermost right returns (represented  by dotted lines).} \label{CHExactnessMnPn}
\end{figure}

Consider the category $C(A^c\pfg)$ of bounded complexes of finitely-generated projective $A^c$-modules up to chain homotopy. This is a monoidal triangulated category, and the inclusion of 
monoidal categories 
\[ A^c\pfg \ \subset \ C(A^c\pfg)
\]
induces an isomorphism of Grothendieck rings 
\[ K_0(A^c\pfg) \ \subset \ K_0(C(A^c\pfg)). 
\]
Thus, $K_0(C(A^c\pfg))\cong \Z[x]$ as a ring. The standard module $M_n$ can be viewed as an object of $C(A^c\pfg)$ due to the finite projective resolution (\ref{CHResMnbyPnTilde}). 

\begin{theorem}[Categorification of the Chebyshev polynomials of the second kind $U_n(x)$]
Standard modules $M_n$ admit finite resolutions by projectives $P_m$ for $m \leq n$, see equation \eqref{CHResMnbyPnTilde}, and can be converted into objects of the homotopy category of finitely-generated projective $A^c$-modules. Their symbols, viewed as elements of the Grothendieck ring $K_0(A^c)\cong \Z[x]$, give the Chebyshev polynomials $U_n(x)$: 
\begin{equation*}
[M_n] = \displaystyle{\sum_{k=0}^{\lfloor \frac{n}{2} \rfloor} (-1)^{k} \sbinom{n-k}{k} [P_{n-2k}]=\sum_{k=0}^{\lfloor \frac{n}{2} \rfloor} (-1)^{k} \sbinom{n-k}{k} x^{n-2k}  = U_n(x)}.
\end{equation*}
\end{theorem}

\section{Relations between two categorifictions, BGG reciprocity and more}


In Section~\ref{subsec-sl2} we briefly recalled the Temperley-Lieb category and
its relation to the representation theory of Lie algebra $sl(2)$ and its quantum deformation. Resolution
(\ref{CHResMnbyPnTilde}) has a counterpart in that theory, as a resolution of the
irreducible  $(n+1)$-dimensional
representation $V_n$ of $sl(2)$ or quantum $sl(2)$ (for generic $q$) by multiples of tensor powers of the two-dimensional
representation $V_1$:

\begin{multline}\label{eq-res-sl2}
0 \ra \ldots \ra (V_1^{\otimes (n-2k)})^{\widetilde{Y}_{n-2k,n}} \stackrel{d_k}\lra
(V_1^{\otimes (n-2(k-1))})^{\widetilde{Y}_{n-2(k-1),n}} \stackrel{d_{k-1}}\lra
\ldots \\
 \ldots \stackrel{d_{2}}\lra (V_1^{\otimes (n-2)})^{n-1 } \stackrel{d_{1}}\lra
V_1^{\otimes n} \ra V_n \ra 0
\end{multline}

The differential is described by the same rules as in (\ref{CHResMnbyPnTilde}),
via composing diagrams with only right returns. In this case the diagrams are viewed
as describing maps between tensor powers of $V_1$ (equivalently, morphisms in
the Temperley-Lieb category). For diagrams with both left and right returns the
composition rules are different from those in $A^c$, but for diagrams with right
returns only there is no change. That (\ref{eq-res-sl2}), for any $q$,
is a resolution of $V_n$ can be proved similar to Proposition~\ref{prop-res-stand},
by using diagrammatics~\cite{frenkel1997canonical,khovanov1998graphical} for the Lusztig dual canonical basis
of tensor powers of $V_1$ instead of the
basis $B^c_n$ of $P_n$. Notice that, unlike the quantum $sl(2)$ case, where $V_n$
is irreducible for generic $q$, standard modules $M_n$ do not even have finite length.

For generic $q$ the category of finite-dimensional representations of quantum $sl(2)$
is semisimple, and resolving objects seems to carry little sense from the homological
viewpoint. Nevertheless, essentially this resolution
was used in~\cite{khovanov2005categorifications} to categorify the colored Jones polynomial, also see \cite{beliakova2008categorification, cooper2010categorification}.
Resolutions of irreducible representations of $sl(n)$
by tensor products of fundamental representations have been studied by Akin, Buchsbaum,
Weyman \cite{akin1985characteristic, akin1982schur} and others, often in the dual Schur-Weyl context, where simple
symmetric group representations are resolved via induced ones, with
applications to  algebraic geometry~\cite{weyman_2003}.


Projective resolution (\ref{CHResMnbyPnTilde}) provides some homological information
about standard modules, observed in the two propositions below.

\begin{prop} Given two standard modules $M_n, M_m$, the $k$-th Ext group for $k
\le \lfloor \frac{n}{2}\rfloor$ is 
$\mathrm{Ext}^k(M_n, M_m)\cong (1_{n-2k}M_m)^{\small{\sbinom{n-k}{k}}}$ and
$$
\dim_{\fieldk} \mathrm{Ext}^k(M_n, M_m) =
\left\{
\begin{array}{cl}
        \small{\sbinom{n-k}{k}\sbinom{\frac{m+n}{2}-k}{m}} & \hbox{ if $m \leq n-2k$ and
$n+m$ is even,} \\
   0 & \hbox{otherwise.}
\end{array}
\right.
$$
\end{prop}

\begin{prop}\label{CHExtMnLm} The $k$-th Ext group for standard and simple modules $M_n, \,L_m$ has dimension
\begin{equation}
\dim_{\fieldk}\mathrm{Ext}^k(M_n, L_m) =
\left\{%
\begin{array}{cl}
        \small{\sbinom{n}{m}} & \hbox{ if $m \leq n, k=\frac{n-m}{2}$,} \\
   0 & \hbox{otherwise.} \\
\end{array}%
\right.
\end{equation}
\end{prop}

\begin{prop}
  Homological dimension of the standard module $M_n$ is $\lfloor
  \frac{n}{2}\rfloor$.
\end{prop}


A finite--dimensional $A^c$--module $M$ has a finite filtration by simple modules $L_n$.
Due to one-dimensionality of $L_n$ the multiplicity of $L_n$ in $M$, denoted by $[M:L_n]$,
equals $\dim_{\fieldk}(1_nM)$. With this observation in mind, we give the following definition.

\begin{defn} For a locally finite-dimensional
$A^c$-module $M$ define the multiplicity of a simple module $L_n$ in $M$ by
 \begin{equation}
   [M:L_n] := \dim_{\fieldk}(1_nM).
   \label{SLARCSMultiplicity}
 \end{equation}\end{defn}

\begin{prop}The multiplicity $[M_m:L_n]$ of $L_n$ in the standard module $M_m$ equals
the number of diagrams in ${}_n^{\,}B^c_m$ with no right returns:
\begin{equation}
  [M_m:L_n] = X_{m,n} = \left\{
              \begin{array}{ll}
                 \frac{2(m+1)}{n+m+1}\sbinom{n}{\frac{n-m}{2}}
                   & \hbox{if $n \geq m$ and  $n-m$ is even,} \\
                0 & \hbox{otherwise.}
              \end{array}
            \right.
\end{equation}
\end{prop}


Add description and \cite{bernstein1976category}

\begin{cor} Indecomposable projective, standard, and simple $A^c$-modules satisfy the
BGG reciprocity property:
\begin{equation} \label{eq-bgg-reciprocity-cheb}
[P_n:M_m] = [M_m : L_n ].
\end{equation}
\end{cor}


Resolution of a simple module $L_n$ by standard
modules $M_m$, which we now describe, is, in a sense, dual to the projective
resolution (\ref{CHResMnbyPn}) of a standard module by projective modules.
This is an infinite to the left resolution, with the $k$-th term consisting
of the standard module $M_{n+2k}$ with multiplicity $|Y_{n,n+2k}|$:

\begin{equation}\label{eq-res-simp-stan}
\ldots \ra M_{n+2k}^{\widetilde{Y}_{n,n+2k}} \stackrel{d_{k}}\lra
M_{n+2(k-1)}^{ \widetilde{Y}_{n,n+2(k-1)}} \stackrel{d_{k-1}}\lra \ldots \stackrel{d_{1}}\lra
M_{n+2}^{ n+1} \ra M_n \ra L_n \ra 0 .
\end{equation}

The differential $d_k$ is a sum of its components
$$d_{k,\beta,\alpha} \ : \ M_{n+2k}^{\alpha} \lra M_{n+2(k-1)}^{\beta},$$
over all $\alpha\in \widetilde{Y}_{n,n+2k}$ and $\beta\in \widetilde{Y}_{n,n+2(k-1)}$.
For each such pair $(\alpha,\beta)$ there exists at most one $i$, $1\le i \le n-1$,
such that $\alpha=\beta b_n^i,$ see Figure~\ref{CHDiffMnLn} for the latter notation. If that's the case, define $j$ to be the order
(counting from the bottom) of the right return of $b_n^i$ when viewed as a right
return of $\alpha$ upon composing with $\beta$.
If such $i$ does not exist, we set
$d_{k,\beta,\alpha}=0$. If it does, we let
$d_{k,\beta,\alpha}(x) = (-1)^{j-1} x \  {}^ib_{n-2},$ for $x\in M_{n+2k}^{\alpha}$.

Notice that any diagram $b\in {}_n^{\,} B^c_{m}$ without right returns
induces a homomorphism of standard modules $M_n\lra M_m$.
Diagram ${}^ib_{n-2}$ has no right returns and induces a  map
from $M_{n+2k}^{\alpha}$ to $M_{n+2(k-1)}^{\beta}$. The summand $d_{k,\beta,\alpha}$
of the differential $d_k$ is just this map, with a sign.

The diagram ${}^ib_{n-2}$ is a reflection of $b^i_n$ about the $y$-axis. In fact,
it would also have been natural to label copies of $M_{n+2k}$ in the resolution
by reflections of diagram in $\widetilde{Y}_{n,n+2k}$, but we did not do this, to
avoid an additional notation.

\begin{prop}
The complex (\ref{eq-res-simp-stan}) is exact.
\end{prop}
\begin{proof}
$M_{n+2k}^{\alpha}$ has a basis of diagrams $\gamma\in B^c_{n+2k}$ without
right returns. Composition $\gamma\alpha'$, where $\alpha'$ is the reflection
of $\alpha$ about a vertical axis, is a diagram in $B^c_n$ without right returns.
The element $d_{k,\beta,\alpha}(\gamma)$, when nonzero, is, up to a sign,
a digram withour right returns in $B^c_{n+2(k-1)}$, and diagrams $\pm d_{k,\beta,\alpha} \beta'$
and $\gamma\alpha'$ are equal. Consequenly, the complex (\ref{eq-res-simp-stan}),
with $L_n$ removed,
decomposes into the direct sum of complexes of vector spaces, over all diagrams
$b \in B^c_n$ without right returns, with basis elements of the underlying vector space
corresponding to pairs $(\alpha,\gamma)$ as above with $\gamma\alpha'=b$.
Each such direct summand is a complex isomorphic to the $r$-th tensor power
of the contractible complex $0\to \fieldk \stackrel{\cong}{\lra} \fieldk \to 0$,
where $r$ is the number of left returns of $b$. Only the diagram $1_n$ leads to
a summand with nontrivial cohomology, which maps isomorphically onto $L_n$,
implying that (\ref{eq-res-simp-stan}) is exact.
\end{proof}


To build a projective resolution of a simple module $L_n$, we start with the
resolution (\ref{eq-res-simp-stan}) of $L_n$ by standard modules and then convert
each standard module $M_{n+2k}$ into its projective resolution
(\ref{CHResMnbyPnTilde}).
These resolutions combine into a bicomplex in the second quadrant of the
plane; Figure~\ref{eq-cd-square} shows one square of the bicomplex.

\begin{figure}[h]
{
\[\begin{CD}
P_{n+2k-2j}^{\widetilde{Y}_{n,n+2k}\times\widetilde{Y}_{n+2k-2j,n+2k}}
@>\displaystyle{}>>
P_{n+2(k-1)-2j}^{\widetilde{Y}_{n,n+2(k-1)}\times\widetilde{Y}_{n+2k-2j,n+2k}} \\
@VV{}V @VV{}V \\
 P_{n+2k-2(j-1)}^{\widetilde{Y}_{n,n+2k}\times\widetilde{Y}_{n+2k-2(j-1),n+2k}}
 @>\displaystyle{}>>
P_{n+2k-2j}^{\widetilde{Y}_{n,n+2(k-1)}\times\widetilde{Y}_{n+2k-2(j-1),n+2k} }  \\
 \end{CD}\]\\}
\caption{An anticommutative square in the bicomplex for a simple
module.} \label{eq-cd-square}
\end{figure}

Horizontal and vertical differentials are defined identically to those in
complexes (\ref{eq-res-simp-stan}) and (\ref{CHResMnbyPnTilde}), correspondingly.
Differential applied to a single term in the direct summand in the upper left
corner is a signed sum of maps in commutative squares
\[\begin{CD}
P_{n+2k-2j}^{\alpha_1\times\alpha_2}
@>\displaystyle{}>>
P_{n+2(k-1)-2j}^{\beta_1\times\alpha_2} \\
@VV{}V @VV{}V \\
 P_{n+2k-2j+2}^{\alpha_1\times\beta_2}
 @>\displaystyle{}>>
P_{n+2k-2j}^{\beta_1\times\beta_2}  \\
 \end{CD}\]
\\
defined in the same way as for complexes (\ref{eq-res-simp-stan}) and
(\ref{CHResMnbyPnTilde}), where $\alpha_1\in \widetilde{Y}_{n,n+2k}$,
$\beta_1 \in \widetilde{Y}_{n,n+2(k-1)}$, etc.

The projective resolution of simple module $L_n$ is obtained by forming the total
complex of this bicomplex. Due to finite-dimensionality of homs between
projective modules $P_m$, this is the unique minimal resolution of $L_n$. Any other
projective resolution is isomorphic to the direct sum of the minimal one with
contractible complexes of the form $0 \to P\stackrel{1}{\lra} P\lra 0$ in
various homological degrees.

\begin{prop}\label{CHExtLm} $\fieldk$-vector space $\mathrm{Ext}^k(L_n, L_m)$
has dimension
$ \sbinom{\frac{3n+2k+m}{4}}{n}\sbinom{\frac{n+2k+3m}{4}}{m}$
if in each of the two binomials $\sbinom{a}{b}$ above both $b, a-b$ are
nonnegative integers. Otherwise $\mathrm{Ext}^k(L_n, L_m)=0$.
\end{prop}

\begin{cor}
  Simple $A^c$-modules $L_n$ have infinite homological dimension.
\end{cor}

\subsection{Koszul algebra structure and relations}

Algebra $A^c$ has a structure of a graded algebra where grading is
given by the total number of left and right returns in a diagram.
The zeroth degree part of $A^c$ is semisimple,
being the direct sum of ground fields $\fieldk 1_n.$ The projective
resolution of $L_n$ is naturally graded, with the $m$-th term generated by
 degree $m$ elements, since the differential is a sum of maps over diagrams
with a single return, all of which have degree one.

\begin{cor}\label{CHKoszul}
The Chebyshev algebra $A^c$ is Koszul.
\end{cor}

We now write $A^c$ via its generators and homogeneous relations. The basis of $A^c$
in degree zero is $\{1_n\}_{n \geq 0}.$ Basis of degree one part of
$1_nA^c1_{n+2}$ is given by diagrams $b^i_{n+2}$, $1 \leq i \leq n+1$, see
Figure~\ref{CHDiffMnLn}. Basis of degree one part of $1_{n+2}A^c1_n$ is
given by diagrams ${}^ib_n$, $1 \leq i \leq n+1$, which are reflections
of $b^i_{n+2}$ relative to a vertical axis, see Figure~\ref{CHDiffMnLn}.
For $m \neq n \pm 2$, the degree one subspace of $1_nA^c1_m$ is trivial.
 \begin{figure}[h]
\includegraphics[width=0.7\textwidth]{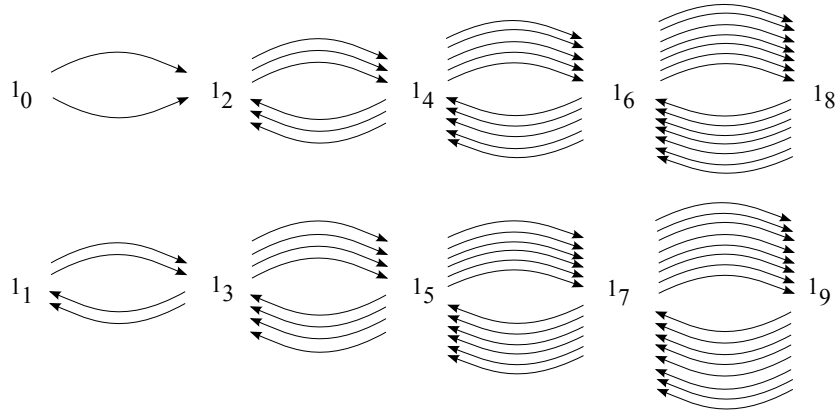}
\caption{Quiver describing homogeneous generators of the algebra $A^c$, with $n+1$ arrows each way between $1_n$ and $1_{n+2}$.} \label{CHQuiver}
\end{figure}
The defining relations, are quadratic, except for the relations involving idempotents $1_n$. The latter relations have  degree zero and one: 
\begin{equation*}\label{CHKoszulRel}
\begin{array}{rl}
  1_n1_m &= \delta_{nm}1_n, \\
  b^i_{n+2}1_m^{} &= \delta_{n+2,m}^{}  b^i_{n+2}, \\
  1_m^{} b^i_{n+2} &= \delta_{m,n}^{}  b^i_{n+2}, \\
  {}^ib_n 1_m &= \delta_{n,m} {}^ib_n, \quad 1_m{}^ib_n^{} = \delta_{n+2,m} {}^ib_n^{}, 
 \end{array}
\quad\quad\quad
\begin{array}{rl}
 b^i_n\  b^j_m &= 0 {\rm \,\ if} \, \ m \neq n+2, \\
 b^i_{n+2} \ {}^jb_m^{}  &= 0 {\rm \, \ if} \, \ m \neq n, \\
 {}^ib_n {}^jb_m &= 0 {\rm \, \ if} \, \ n \neq m+2, \\
{}^ib_n^{}  b^j_{m+2} &= 0 {\rm \, \ if} \, \ n \neq m.\\
\end{array}
\end{equation*}

Diagrammatically, these relations come from the conditions that the product
of diagrams is zero if the number of endpoints does not match. We can represent
these generators as arrows in the quiver which is a disjoint union of quivers
for even and odd values of $n$ in $1_n$, see Figure~\ref{CHQuiver}.
\begin{figure}[h!]
\includegraphics[width=0.7\textwidth]{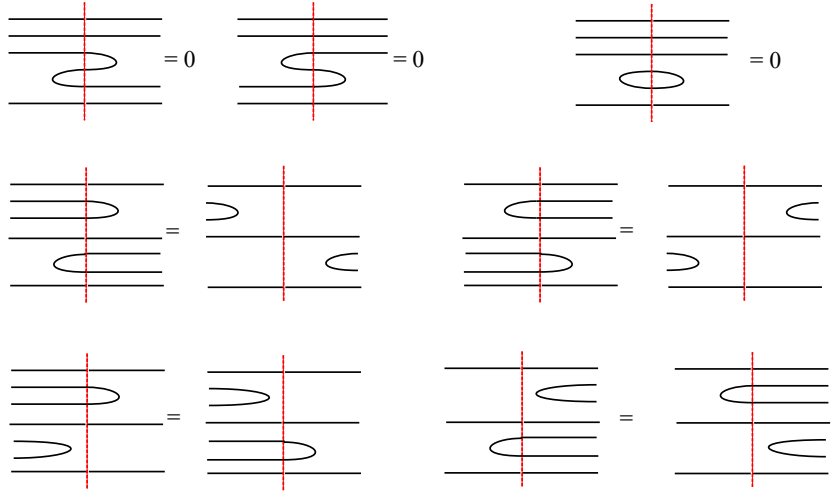}
\caption{Chebyshev quadratic relations from \eqref{CHKG1}-\eqref{CHKG5}: from left to right and top to bottom. } \label{CHKoszulGenuineF}
\end{figure}

The following genuine quadratic relations are
shown schematically in Figure~\ref{CHKoszulGenuineF}:
\begin{eqnarray}\label{CHKG1}
 b^i_{n+2} \ {}^{i \pm 1}b_n^{} & = &0, \\\label{CHKG2}
  b^i_{n+2} \ {}^{i}b_n^{} & = &0, \\\label{CHKG3}
 b^i_{n+2}\ {}^jb_n & = & \begin{cases}  
{}^{j}b_{n-2}\ b^{i-2}_n &{\rm{if} } \ i \geq j+2 \\
 {}^{j}b_{n-2}^{}\ b^{i-2}_n  &{\rm{if }} \ j \geq i+2,
 \end{cases}\\\label{CHKG4}
{}^ib_{n+2} \ {}^jb_{n}^{} & = & {}^{j+2}b_{n+2}\ {}^i b_{n}    {\rm \ \, if}\ j \geq i+2,\\\label{CHKG5}
b^i_{n+2} \ b^{j}_{n} & = & b^{j-2}_{n+2}\ b^i_{n}   {\rm \ \, if} \ j \geq i+2.
                      \end{eqnarray}


\subsection{Approximations of the identity via truncation functors.}
Following the notation ${}_n^{\,}B^c_m(k)$ for the diagrams of width $k$, let $A^c(\le k) \subset A^c$ denote the two-sided ideal of $A^c$ generated by diagrams with at most $k$ through strands.  Notice that $A^c(\le k) \subset A^c(\le k+1)$ and that $\displaystyle{ \cup_{k \geq 0}A^c(\le k)=A^c.}$ For $k \geq 0$, define a right exact functor $F_k : A^c \dmod \ra A^c\dmod$ by $F_k(M)=A^c(\le k) \otimes_{A^c} M.$

\begin{prop} Functor $F_k$ acts on indecomposable projective modules
by the identity, $F_k(P_n)=P_n $ for $n \leq k$,  and  
$F_k(P_n)=P_n (\le k):=A^c(\le k)1_n$ for $n > k$.  For standard modules $M_n \in A^c \dmod$ the action is $F_k(M_n)=M_n$  for $n \leq k$ and $F_k(M_n)=0$ for $n > k$. 
\end{prop}

The proof is straightforward. Note, in particular, that $M_n$ has a basis of diagrams with exactly $n$ through strands.
Modules $F_k(P_n)=P_n(\le k)$ for $n>k$ have finite homological dimension,
since they admit finite filtrations with successive quotients isomorphic to standard modules $M_{n-2m}$ for $n-2m \leq k$. In particular, $P_n(\le k)$ has a finite length resolution by finitely-generated projective, and $F_k$ is a well-defined functor on the category of bounded complexes of finitely-generated projective $A^c$-modules. 

$\square$

\begin{prop} Derived functors of the functor $F_k$ applied to a standard module are $L^iF_k(M_n)=M_n$ if $n\leq k$ and $i=0$, otherwise $L^iF_k(M_n)=0$. 
\end{prop}

\begin{proof} To compute the derived functor $LF_k$ on $M_n$ we apply $F_k$ to the terms of the projective resolution (\ref{CHResMnbyPnTilde}), removing non-projective term $M_n$ on the far right of the diagram. If $n\le k$, $F_k$ acts as identity on all
terms of the resolution. Consequently $L^0F_k(M_n)\cong M_n$ and 
$L^iF_k(M^n)=0$ for $i>0$ in this case. If $n > k$, applying $F_k$ to all terms of the projective resolution (\ref{CHResMnbyPnTilde}) results in an exact complex. In the standard diagram bases of projective modules in this resolution, applying $F_k$ removes all diagrams of width greater than $k$. The remaining diagrams, of width at most $k$, constitute an exact complex of $A^c$-modules. Note that only diagrams of width $n$ (modulo the span of those of smaller width) constitute a non-exact complex, whose homology is, naturally, $M_n$. Thus, for $n>k$, the derived functor $LF_k(M_n)=0$.
\end{proof}

Recall that the Grothendieck group $K_0(A^c)$ of finitely-generated projective $A^c$-modules has a basis $\{[P_n]\}_{n\ge 0}$ of symbols of indecomposable projective modules. Monoidal structure of $A^c$ takes products of projectives to projectives and induces multiplication on $K_0(A^c)$. The latter can naturally  be identified with $\Z[x]$, with $x^n=[P_n]$. 

Thus, symbols of  indecomposable projective modules correspond to the monomials $[P_n]=x^n$, while 
symbols of 
standard modules correspond to Chebyshev polynomials, $[M_n]=U_n$. 

Functor $F_k$ descends to an operator on the Grothendieck group $K_0(A^c)$, denoted
by $[F_k].$ This operator acts by 
$$[F_k](x^n)=[F_k]([P_n]) = x^n, \ \  \mathrm{if} \ \ n\le k,$$ 
and 
$$[F_k](U_n)=[F_k]([M_n])=[LF_k(M_n)]=0 \  \ \mathrm{for}  \ \ n>k.$$ Thus, $[F_k]$ acts by identity on the subspace of polynomials of degree at most 
$k$ and by zero on the linear span of Chebyshev polynomials $U_n$ for $n>k$. It's a reproducing kernel, and the projection operator onto the subspace spanned by the first  $k+1$ Chebyshev polynomial orthogonally to the subspace spanned by $U_n$ for $n > k$. We can think of $[F_k]$ as an approximation to the 
identity operator, which gets better as $k$ goes to infinity. Likewise, the truncation functor $F_k$ and its derived functors can be 
thought of as approximations to the identity functor. In particular $F_k$ acts as the identity functor on the full subcategory of the triangulated category generated by modules $P_n$ for $n\le k$, while annihilating the full subcategory of complexes of projectives generated by resolutions of $M_n$ for $n>k$.


\subsection{Restriction and induction functors} 
For a unital inclusion $\iota:B\subset A$ of arbitrary rings the
induction functor $Ind: B\dmod \ra A\dmod$,  defined by $Ind(M)=A
\otimes_B M$, is left adjoint to the restriction functor $Res: A\dmod \ra B\dmod$
$$ \Hom_A(Ind(M),N) \cong \Hom_B(M,Res(N)). $$
 
If the inclusion is non-unital, i.e. $\iota$ takes the unit
element of $B$ to an idempotent $e\not= 1$ of $A$, the restriction
functor needs to be redefined. In this case, to a $B$-module $N$ the  restriction functor assigns an
$eAe$-module $N/(1-e)N$ and then restricts the action to $B$. 

The
induction functor is defined as before: \begin{equation*}
Ind(M)=A\otimes_B M \cong Ae \otimes_B M \oplus A(1-e)\otimes_B M =
Ae\otimes_B M\end{equation*}
and the induction is still left adjoint to the restriction. A
similar construction works for non-unital $B$ and $A$ equipped
with systems of idempotents.

We consider the map from  Chebyshev diagrams in ${}_n^{}B^c_m$ to those in ${}_{n+1}^{}B^c_{m+1}$ given by adding a horizontal line above a diagram. This map of diagrams respects composition and induces an
inclusion of idempotented algebras $\iota: A^c \hookrightarrow A^c$ such that $\iota(1_n)=1_{n+1}.$
The inclusion gives rise to induction and
restriction endofunctors of $A^c\dmod$, which we denote $\Ind$ and $\Res$. Note that, 
$\Res$ is exact and $\Ind$ is right exact. 


The induction functor takes $P_n$ to $P_{n+1}$.




On the level of Grothendieck group, the induction functor descends to the operator of
multiplication by $x$. 
The action of the induction functor $\Res$ is more complicated and does not seem to admit an elegant description, since $\Res(P_n)$ is neither projective nor finitely generated. 



\bibliographystyle{abbrv}
\bibliography{refCheb}

\end{document}